\documentclass[11pt,reqno]{amsart}
\usepackage[body={7in,9in},centering]{geometry}
\usepackage{amsmath,amssymb,amsthm,mathrsfs}
\usepackage{color,xcolor}
\definecolor{cobalt}{RGB}{61,99,181}
\usepackage[colorlinks,citecolor=cobalt,linkcolor=cobalt]{hyperref}
\usepackage{cite}
\usepackage{float}
\usepackage{exscale}
\usepackage{relsize}
\usepackage{graphicx}
\usepackage{tikz}

\newtheorem{thm}{Theorem}[section]

\newtheorem{lem}[thm]{Lemma}

\numberwithin{equation}{section}

\date{\today}

\usepackage{etoolbox}

\makeatletter

\newcommand{\Rmnum}[1]{\expandafter\@slowromancap\romannumeral #1@}
\patchcmd{\thebibliography}{\section*{\refname}}{\section*{References}}{}{}
\patchcmd{\thebibliography}{\MakeUppercase\refname}{References}{}{}
\makeatother

\def\D{\mathbb{D}}
\def\T{\mathbb{T}}

\newcommand{\C}{\mathbb{C}}

\begin{document}

\title[The spectral structure of Toeplitz operators on $L_h^2$]{The spectral structure of Toeplitz operators with polynomial symbols on the harmonic Bergman space}

\author[Huang]{Mingmei Huang}
\address{
College of Mathematics and Statistics, Chongqing University, Chongqing, 401331, P. R. China}
\email{huangmm0909@163.com}

\author[Wang]{Jiawei Wang}
\address{School of Mathematical Sciences, Inner Mongolia University, Hohhot, 010021, P. R. China}
\email{jiaweiwang1516@163.com}

\author[Zhao]{Xianfeng Zhao}
\address{
College of Mathematics and Statistics, Chongqing University, Chongqing, 401331, P. R. China}
\email{xianfengzhao@cqu.edu.cn}

\keywords{Toeplitz operator; harmonic Bergman space; spectrum; Weyl's theorem.}
	
\subjclass[2010]{47B35, 47A10}

\begin{abstract}
Using the theory of symmetric polynomials, we prove that the spectrum of every Toeplitz operator with an analytic polynomial symbol on the harmonic Bergman space can be represented as the union of a closed curve and a finite set of at most $(n-1)2^{n-1}$ isolated eigenvalues, where $n$ is the degree of the symbol. This implies that Weyl's theorem holds for such operators.
\end{abstract}
\maketitle
	

\section{Introduction}\label{Intro}
	Let $dA = \frac{1}{\pi}dxdy$ denote the normalized area measure on the open unit disk $\mathbb{D}$, and let $L^2(\mathbb{D}, dA)$ be the Hilbert space of square-integrable functions on $\mathbb{D}$ with respect to $dA$. The Bergman space $L_a^2$ is the closed subspace of $L^2(\mathbb D,dA)$ consisting of analytic functions. It is well known that the reproducing kernel (at $w\in \mathbb D$) for the Bergman space $L_a^2$ is given by
	\[
	K_w(z) = \frac{1}{(1-\overline{w}z)^2}, \ \ \ \ z \in \mathbb{D}.
	\]
	The harmonic Bergman space $L_h^2$ is defined as the closed subspace of $L^2(\mathbb{D}, dA)$ consisting of  complex-valued harmonic functions on $\D$. Clearly, we can decompose the harmonic Bergman space as $$L_h^2=zL_a^2\oplus \overline{L_a^2},$$
	where $\overline{L_a^2}$ denotes  the complex conjugate of $L_a^2$. Therefore,
	$$\big\{\sqrt{n+1}z^n: n=1,2,\cdots\big\}\cup\big\{\sqrt{n+1}\overline z^{n}:n=0, 1, 2,\cdots\big\}$$
	forms an orthonormal basis for $L_h^2$. The harmonic Bergman space is also a reproducing kernel Hilbert space and the reproducing kernel is given by
	$$R_w(z)=\frac{1}{(1-\overline{w}z)^2}+\frac{1}{(1-w\overline{z})^2}-1, \ \ \ \  z\in \mathbb D.$$
	
	Let $P$ and $P_h$ be the orthogonal projections from $L^2(\D,dA)$ onto $L_a^2$ and $L_h^2$, respectively. Let $L^\infty(\mathbb D)= L^\infty(\mathbb D, dA)$ denote the Banach algebra of essentially bounded functions on $\mathbb D$ with respect to the measure $dA$.
For $\varphi\in L^\infty(\D)$ , the Toeplitz operator $T_{\varphi}$ with symbol $\varphi$ on $L_a^2$ is defined by
	\[
	T_{\varphi}f=P(\varphi f),\qquad f\in L_a^2.
	\]
	Similarly, the Toeplitz operator $\mathbb T_\varphi$ on $L_h^2$ with symbol $\varphi$ is defined by
	\[
	\mathbb T_\varphi u=P_h(\varphi u),\qquad u\in L_h^2.
	\]
	For further information about the Bergman and harmonic Bergman spaces, and Toeplitz operators on these two function spaces, the reader is referred to \cite{Zhu}, \cite{CL} and \cite{GZ}.

The purpose of this paper is to investigate the spectral structure of Toeplitz operators with general analytic polynomial symbols on the harmonic Bergman space, as well as the Weyl spectrum and Weyl's theorem for such operators. The spectral theory of Toeplitz operators on various function spaces has attracted considerable attention. According to the well-known Widom theorem, the spectrum of a Toeplitz operator on the Hardy space is always connected (see \cite[Corollary 7.46]{Dou}). However, the spectral structure  of Toeplitz operators on the Bergman space exhibits  peculiar phenomena: even for harmonic polynomial symbols, the spectrum of a Toeplitz operator may be disconnected (see \cite[Theorem 4.1]{GZZ}). As for Toeplitz operators on the harmonic Bergman space, their spectral properties turn out to be even more unexpected: there exists a Toeplitz operator with a certain quadratic analytic polynomial as its symbol whose spectrum is disconnected (see \cite[Proposition 3.3]{WangZhao}). In Section \ref{S3}, we investigate the topological structure of the spectra of Toeplitz operators with polynomial symbols on the harmonic Bergman space and discuss Weyl's theorem for such operators. To this end, we first recall the definition of the Weyl spectrum and review the relevant literature.

The Weyl spectrum plays an important role in the study of the spectral theory of linear operators.  For a bounded linear operator $T$ on a complex Hilbert space $\mathcal{H}$, the Weyl spectrum $\omega(T)$ of $T$ is defined by
	\[
	\omega(T):=\bigcap_{K\ \text{\ compact}}\sigma(T+K),
	\]
	where the intersection is taken over all compact operators $K$ on $\mathcal{H}$, and $\sigma(T+K)$ denotes the spectrum of $T+K$. In \cite{S}, Schechter provided an equivalent formulation of the Weyl spectrum, showing that
	\begin{align*}
		\omega(T) = \sigma_{\text{e}}(T) \cup \{\lambda \in \mathbb{C}: \lambda\notin\sigma_{\text{e}}(T)\  \ \text{and}\ \ \text{index}(T - \lambda I) \neq 0\},
	\end{align*}
	where $\sigma_{\mathrm{e}}(T)$ denotes the essential spectrum of $T$ and $\operatorname{index}(T - \lambda I)$ stands for the Fredholm index of the operator $T-\lambda I$.
	
The systematic study of compact perturbations of linear operators dates back to  Weyl \cite{W}, who proved that for a self-adjoint operator $T$ on a Hilbert space, a spectral point $\lambda$ belongs to $\sigma(T+K)$ for every compact operator $K$ if and only if $\lambda$ is not an isolated eigenvalue of finite geometric multiplicity of $T$.
Extending the classical result of Weyl \cite{W} for normal operators, Coburn \cite{C} showed that if \(T\) is any hyponormal operator or Toeplitz operator on the Hardy space, then
\begin{equation}\label{WT}
\omega(T) = \sigma(T)\backslash\pi_{00}(T),
\end{equation}
where $\pi_{00}(T)$ denotes the set of isolated eigenvalues of $T$ of finite geometric multiplicity. In general, if a bounded linear operator \(T\) satisfies (\ref{WT}), one says that Weyl's theorem holds for \(T\).
In 1968, Istr\u{a}\c{t}escu \cite{Ist} showed that Weyl's theorem holds for a certain class of operators that are not necessarily hyponormal.
Later, Berberian \cite{B, Ber} proved that Weyl's theorem holds for a class of operators that includes both the hyponormal operators and the class of operators considered by Istr\u{a}\c{t}escu \cite{Ist}. These works gave rise to studies on  Weyl's theorem for algebraically hyponormal operators \cite{HL}, algebraically paranormal operators \cite{HC}, Banach-space operators with the single-valued extension property \cite{CH}, paranormal operators \cite{UA}, algebraically quasi-class A operators \cite{AH}, $(2\times 2)$ operator matrices \cite{Cao},
and quasi-2-expansive operators \cite{SJ}.

However, apart from Toeplitz operators on the Hardy space and weighted shift operators, most scholars have only considered Weyl's theorem for abstract operators. In the setting of the Bergman space, Guo, Zhao and Zheng  proved that the Toeplitz operator with the harmonic symbol $\overline{z}+q$  satisfies Weyl's theorem whenever $q$ belongs to the disk algebra; see \cite[Theorem 5.2]{GZZ}. In contrast, in the case of the harmonic Bergman space, Wang and Zhao  established  that Weyl's theorem holds for the Toeplitz operator $\mathbb T_p$ when $p$ is an analytic polynomial of degree at most $3$; see \cite[Theorem 3.1]{WangZhao}.  The case of general analytic polynomials was left open.
	
In the present paper, using the theory of symmetric polynomials, we prove that for any analytic polynomial $p$, the spectrum of the Toeplitz operator with symbol $p$ on the harmonic Bergman space can be written as the union of a closed curve and finitely many isolated eigenvalues. This yields that such Toeplitz operators satisfy Weyl's theorem. It is worth noting that this result improves Theorem 3.1 of \cite{WangZhao}
in an essential way and is proved using a different technique; see
Theorem \ref{thm:main} in Section \ref{S3} for details.

	\section{Preliminaries}\label{S2}
	
In this section, we recall some notation and  two known results concerning Toeplitz operators on the harmonic Bergman space, which will be used later. The first one, due to Guo and Zheng, is particularly useful for studying Toeplitz operators with continuous symbols on the harmonic Bergman space $L_h^2$.
	
	\begin{thm}\emph{(\hspace{0.01em}\cite[Theorem 3.4]{GZ})}\label{lemGZ}
	Let $\varphi\in C(\overline{\D})$. Then the essential spectrum of $\T_\varphi$ is given by $\sigma_{\mathrm{e}}(\T_\varphi)=\varphi(\partial \D)$, where $\partial \mathbb D$ denotes the unit circle.  Moreover, if  $\T_\varphi$ is a Fredholm operator on $L_h^2$, then $\mathrm{index}(\T_\varphi)=0$.
	\end{thm}

We denote by $H^\infty$  the Banach algebra of bounded analytic functions on the disk $\mathbb D$.
For $f\in L^{2}(\D, dA)$, we define $$\widehat{f}(z)=f(\bar{z})\ \ \
 \text{and}\ \ \ \ f^*(z)=\overline{f(\bar{z})}.$$
	Let $U:L^2(\D, dA)\to L^2(\D, dA)$ be the unitary operator defined as follows:  $$Uf(z)=f(\overline{z})=\widehat{f}(z),\ \ \ \ z\in \mathbb D.$$
	Define the unitary operator $$W: L_h^2\longrightarrow zL_a^2\oplus L_a^2$$ by $$W=\begin{bmatrix}
		I & O \\
		O & U
	\end{bmatrix}.$$
	
	With the above notation, Wang and Zhao  established the following matrix representation of $W\mathbb T_\varphi W^*$ with $\varphi\in H^\infty$, which will be needed in the proof of our main result.
	\begin{lem}\label{lemWZ}\emph{(\hspace{0.01em}\cite[Lemma 2.2]{WangZhao})}
		 Let $\varphi\in H^{\infty}$. On the space $zL_a^2\oplus L_a^2$, we have $$W\mathbb T_\varphi W^*=\begin{bmatrix}
			T_\varphi & \Gamma_\varphi-(1\otimes\varphi^*) \vspace{2mm}\\
			{O} & T_{\widehat{\varphi}}
		\end{bmatrix},$$
		where $\Gamma_\varphi$ is the small Hankel operator on $L_a^2$ defined by $$\Gamma_\varphi (f)=P(\varphi Uf),$$ and $1\otimes\varphi^*$ is the rank-one operator on $L_a^2$.

Consequently, $ \mathbb T_\varphi$ is invertible on the harmonic Bergman space  $L_h^2$ provided that $\varphi$ is invertible in $H^\infty$.
	\end{lem}

	\section{The main result}\label{S3}
	This section is devoted to the study of spectral structure of Toeplitz operators with analytic polynomial symbols on the harmonic Bergman space $L_h^2$. The main result of this paper is stated as follows.
	\begin{thm}\label{thm:main}
		Let $p$ be an analytic polynomial of degree $n\geqslant 1$.  Then there exists a finite set
		$\Lambda_p\subset\C \backslash p(\partial\D)$ such that
		\[
		\sigma(\mathbb T_p)=p(\partial\D)\cup\Lambda_p
		\quad\text{and}\quad
		\Lambda_p=\pi_{00}(\mathbb T_p),
		\]
where $\pi_{00}(\mathbb T_p)$ denotes the set of isolated eigenvalues of $\mathbb T_p$ of finite geometric multiplicity.  Moreover, the cardinality of $\Lambda_p$ is at most $(n-1)2^{n-1}$.

As a consequence,
		 Weyl's theorem holds for the Toeplitz operator $\mathbb T_p$, i.e.,
		\[
		\omega(\mathbb T_p)=\sigma(\mathbb T_p) \backslash\pi_{00}(\mathbb T_p).
		\]
	\end{thm}
	
	\begin{proof}
If $p$ is a linear polynomial, then we have by \cite[Theorem 3.1]{WangZhao} that $\sigma(\mathbb T_p)=p(\partial \mathbb D)$. In this case, $\Lambda_p=\varnothing$. So, we may assume that $n\geqslant 2$ and
		\begin{equation*}
			p(z)=a_0+a_1z+a_2z^2+\cdots+a_nz^n,\ \ \ \ a_n=1.
		\end{equation*}
		Next, we will show that the Weyl spectrum of $\mathbb T_p$ satisfies
		\begin{equation*}
			\omega(\mathbb T_p)=\sigma(\mathbb T_p) \backslash\pi_{00}(\mathbb T_p),
		\end{equation*}
		where
		\begin{align*}
			\omega(\mathbb T_p) = \sigma_{\text{e}}(\mathbb T_p) \cup \{\lambda \in \mathbb{C}: \lambda\notin \sigma_{\text{e}}(\mathbb T_p)\  \ \text{and} \  \ \text{index}(\mathbb T_p - \lambda I) \neq 0\}.
		\end{align*}
		From Theorem \ref{lemGZ}, we conclude that $\sigma_{\mathrm e}(\T_p)=p(\partial\D)$ and $\T_{p-\lambda}$ is a Fredholm operator with index $0$ for every $\lambda\notin p(\partial\D)$. Thus, $$\omega(\T_p)=\sigma_{\mathrm{e}}(\T_p)=p(\partial\D).$$
		Noting that a Fredholm operator with Fredholm index zero is not invertible if and only
		if its kernel is nontrivial, the spectrum of $\T_p$ can be written as
		\begin{equation}\label{eq:spectrum}
			\sigma(\mathbb T_p)=p(\partial\D)\cup\bigl\{\lambda\in\C : \lambda \notin p(\partial\D)\  \ \text{and}\ \
			\lambda\in \sigma_{\mathrm{p}}(\mathbb T_p)\bigr\},
		\end{equation}
		where $\sigma_{\mathrm{p}}(\mathbb T_p)$ is the point spectrum of $\mathbb T_p$.
		
		Denote $$\Lambda_p:=\bigl\{\lambda\in\C :\lambda\notin  p(\partial\D)\  \ \text{and}\ \
		\lambda\in \sigma_{\mathrm{p}}(\mathbb T_p)\bigr\}.$$ From (\ref{eq:spectrum}), we conclude that the isolated points of $\sigma(\T_p)$ (if any) are contained in $\Lambda_p$, and hence we have that $\pi_{00}(\mathbb T_p)\subset\Lambda_p$. Therefore, to show Weyl's theorem holds for $\T_p$, we need only to show that $\Lambda_p$ is a finite set.

		If $\lambda\notin p(\partial\mathbb D)$ and $p-\lambda$ has a
		multiple zero $\gamma$, then
		\[
		p(\gamma)=\lambda
		\quad\text{and}\quad
		p'(\gamma)=0.
		\]
		Since the polynomial $p'$ has only finitely many zeros, every such $\lambda$ belongs to the finite set
		\[
		\mathcal C_p:=\{p(\gamma):p'(\gamma)=0\}.
		\]
		Therefore, there are only finitely many  $\lambda\in\C\backslash p(\partial\D)$ for which
		$p-\lambda$ has a multiple zero. It remains to consider those
		$\lambda\in\Lambda_p$ for which all the zeros of $p-\lambda$ are simple.

		Let $\lambda$ be a point in $\Lambda_p$ such that all the zeros of $p-\lambda$ are simple. Note that $\lambda\in \sigma_{\mathrm{p}}(\T_p)$ if and only if $0\in \sigma_{\mathrm{p}}(\T_{p-\lambda})$, which is equivalent to $0\in \sigma_{\mathrm{p}}(W\T_{p-\lambda}W^*)$. Here, $W$ is the unitary operator defined in Section \ref{S2}. Let $\psi(z):=p(z)-\lambda$. Then it follows from Lemma \ref{lemWZ}  that there exist $f\in zL_a^2$ and $g\in L_a^2$, not both zero, such that
		\begin{equation}\label{eq:block-system}
			\begin{cases}
				\psi(z)f(z)+P\bigl[\psi(z)g(\overline z)\bigr]
				-\langle g,\psi^*\rangle=0,\vspace{2mm}\\
				T_{\widehat\psi}g=0.
			\end{cases}
		\end{equation}
		We first deal with the second equation in (\ref{eq:block-system}). From the proof of \cite[Theorem 3.1]{WangZhao}, we have
		$$\mathrm{dim\ ker}(T_{\widehat{\psi}})=\mathrm{the \ number \ of \ zeros  \ of} \ \overline{\widehat{\psi}}\ \mathrm{in}\ \mathbb D.$$
		If $\psi$ has no zeros in the open disk $\mathbb D$, then it has no zeros in $\overline{\D}$, since $\lambda\notin p(\partial \D)$. Hence $\psi$ and $\frac{1}{\psi}$ belong to $H^\infty$. In this case,  Lemma \ref{lemWZ} tells us that $\T_\psi$ is invertible on $L_h^2$.  Hence $\lambda\notin\Lambda_p$, which is a contradiction.

In the sequel, we assume that $p-\lambda$ has $N$ distinct zeros in $\D$, denoted by $\alpha_1, \cdots, \alpha_N$, where $1\leqslant N\leqslant n$. Note that
		\[T_{\widehat\psi}K_{\overline{\alpha}_j}=\overline{\psi^*(\overline{\alpha}_j)}K_{\overline{\alpha}_j}=\psi(\alpha_j)K_{\overline{\alpha}_j}=0,\ \ \ \ j=1,2,\cdots, N.\]
		Since $\mathrm{dim\ ker}(T_{\widehat{\psi}})=N$ and reproducing kernels at distinct points are linearly independent, the solution for the second equation in (\ref{eq:block-system}) has the form:
		\begin{equation}\label{eq:g}
			g(z)=\sum_{j=1}^{N}c_jK_{\overline{\alpha}_j}(z)
			=\sum_{j=1}^{N}\frac{c_j}{(1-\alpha_jz)^2},
			\qquad c_j\in\C.
		\end{equation}
		It is easy to see that
		\[\langle K_{\overline{\alpha}_j},\psi^*\rangle=\overline{\psi^*(\overline{\alpha}_j)}=\psi(\alpha_j)=0,\]
		so we conclude that  $\langle g,\psi^*\rangle=0$. Thus the first equation in (\ref{eq:block-system}) becomes:
		$$\psi(z)f(z)+P\bigl[\psi(z)g(\overline z)\bigr]=0.$$
		
		Next, we are going to compute $P\bigl[\psi(z)g(\overline z)\bigr]$. Recall that the reproducing kernel for the Bergman space  $L_a^2$ can be expressed as
		\[K_z(w)=\frac{1}{(1-w\overline z)^2}=\sum_{\ell=0}^{\infty}(\ell+1)w^\ell\overline z^{\,\ell}.\]
		It follows from (\ref{eq:g}) that
		\begin{align*}
			P\bigl[\psi(z)g(\overline z)\bigr]
			=& P\Big[\psi(z)\sum_{j=1}^{N}\frac{c_j}{(1-\alpha_j\overline{z})^2}\Big]\\
			=& \sum_{j=1}^{N}c_jP\Big[\frac{p(z)-\lambda}{(1-\alpha_j\overline{z})^2}\Big].
		\end{align*}
Since $P$ is bounded on $L^2(\mathbb D, dA)$, we have that
		\begin{align*}
			P\left[\frac{z^k}{(1-w\overline z)^2}\right]&=\sum_{\ell=0}^{\infty}(\ell+1)w^\ell	P(z^k\overline z^{\,\ell})\\
			&=\frac{1}{k+1}\sum_{\ell=0}^{k}
			(\ell+1)(k-\ell+1)w^\ell z^{k-\ell},
		\end{align*}
		where the second equality is due to
		\begin{equation*}
			P(z^k\overline z^{\,\ell})=
			\begin{cases}
				\displaystyle\frac{k-\ell+1}{k+1}z^{k-\ell},&0\leqslant \ell\leqslant k,\\[6pt]
				0,&\ell>k.
			\end{cases}
		\end{equation*}
		Consequently, we obtain that
		\begin{align*}
			P\left[\frac{z^k-w^k}{(1-w\overline z)^2}\right]
			&=\frac{1}{k+1}\sum_{\ell=0}^{k-1}
			(\ell+1)(k-\ell+1)w^\ell z^{k-\ell}\\
			&=\frac{z}{k+1}\sum_{\ell=0}^{k-1}
			(\ell+1)(k-\ell+1)w^\ell z^{k-\ell-1}.
		\end{align*}
		Define the following  polynomial in two variables:
		\begin{equation}\label{eq:h}
			h(z,w):=
			\sum_{k=1}^{n}\frac{a_k}{k+1}
			\sum_{\ell=0}^{k-1}(\ell+1)(k-\ell+1)
			z^{k-\ell-1}w^\ell.
		\end{equation}
		Then the linearity of $P$ gives
		\begin{equation*}
			P\left[\frac{p(z)-p(w)}{(1-w\overline z)^2}\right]=zh(z,w),
			\qquad w\in\D.
		\end{equation*}
		Using $p(\alpha_j)=\lambda$, we obtain that
		\begin{align*}
			P\bigl[\psi(z)g(\overline z)\bigr]
			=\sum_{j=1}^{N}c_j
			P\left[\frac{p(z)-p(\alpha_j)}
			{(1-\alpha_j\overline z)^2}\right]=z\sum_{j=1}^{N}c_jh(z,\alpha_j).
		\end{align*}
		This yields that the solution of the first equation in (\ref{eq:block-system}) is given by
		\begin{equation}\label{eq:f}
			f(z)=-z\,
			\frac{\displaystyle\sum_{j=1}^{N}c_jh(z,\alpha_j)}
			{p(z)-\lambda}.
		\end{equation}
		Since $f\in zL_a^2$, we obtain that $\frac{f}{z}$ is analytic and square integrable on $\D$.
		
		Recall that \(N\) denotes the number of zeros of \(p-\lambda\) in \(\mathbb D\), and \(n\) denotes the degree of \(p\). If $N=n$, then the polynomial in the numerator of (\ref{eq:f}) has degree at most $(n-1)$ and vanishes at the $n$ distinct zeros of $p-\lambda$. Hence, we have
		$$\sum_{j=1}^{n}c_jh(z,\alpha_j)\equiv 0.$$
After straightforward calculations and rearrangements, the above system of equations can be written as
		\begin{equation*}
			\begin{cases}
				c_1+c_2+\cdots+c_n=0,\\
				c_1\alpha_1+c_2\alpha_2+\cdots+c_n\alpha_n=0,\\
				\cdots\cdots\\
				c_1\alpha_1^{n-1}+c_2\alpha_2^{n-1}+\cdots+c_n\alpha_n^{n-1}=0.
			\end{cases}
		\end{equation*}	
		It is clear that the determinant of the coefficient matrix of $\{c_1,c_2,\cdots,c_n\}$ is a Vandermonde determinant, which is equal to $$\prod\limits_{1\leqslant j<k\leqslant n}(\alpha_k-\alpha_j).$$
		Since $\alpha_k\neq \alpha_j$ for $k\neq j$, the coefficient matrix is invertible and hence $$c_1=c_2=\cdots=c_n=0.$$
		Thus $f=g=0$, which is a contradiction.

 From the preceding discussion, we find that it suffices to consider the case that $1\leqslant N\leqslant n-1$. In this case, the condition $f\in zL_a^2$ yields the following system of equations:
		\begin{equation*}\label{eq:linear-system}
			\sum_{j=1}^{N}c_jh(\alpha_i,\alpha_j)=0,  \ \ \ \ \ i=1, 2, \cdots, N,
		\end{equation*}
where $h$ was defined in (\ref{eq:h}).
		Therefore, $\lambda$ is an eigenvalue of $\T_p$ if and only if the above system admits a nontrivial solution, which is equivalent to the condition
		\begin{equation}\label{det}
			\det\bigl[h(\alpha_i,\alpha_j)\bigr]_{i,j=1}^{N}=0.
		\end{equation}
		Here,  $\bigl[h(\alpha_i,\alpha_j)\bigr]_{i,j=1}^{N}$ denotes the $N$-th order matrix whose elements are $h(\alpha_i,\alpha_j)$:
		$$\begin{bmatrix}
			h(\alpha_1,\alpha_1) & h(\alpha_1,\alpha_2) & \cdots & h(\alpha_1,\alpha_N) \\
			h(\alpha_2,\alpha_1) & h(\alpha_2,\alpha_2) & \cdots & h(\alpha_2,\alpha_N) \\
			\vdots & \vdots & \ddots & \vdots \\
			h(\alpha_N,\alpha_1) & h(\alpha_N,\alpha_2) & \cdots & h(\alpha_N,\alpha_N)
		\end{bmatrix}.$$
		Thus, to show there are finitely many values of $\lambda\in \Lambda_p$ for which $p-\lambda$ has only simple zeros,  it is enough to show that there are finitely many points in $\C\backslash p(\partial\D)$, still denoted by $\lambda$, that satisfy  $\lambda=p(\alpha_j)$ for $j\in\{1,\cdots,N\}$ and for which  (\ref{det}) holds.
		
		Since the degree of $p-\lambda$ is $n$, let
		$\beta_1,\cdots,\beta_n$ be all the zeros of $p-\lambda$ in the complex plane $\mathbb C$, and set
		\begin{equation}\label{FN}
			F_N(\beta_1,\cdots,\beta_n):=
			\prod_{\substack{A\subset\{1,\cdots,n\}\\ |A|=N}}
			\det\bigl[h(\beta_i,\beta_j)\bigr]_{i,j\in A},
		\end{equation}
		where $|A|$ denotes the cardinality of the set $A$. If (\ref{det}) holds, then one of the factors in (\ref{FN}) vanishes, and hence $$F_N(\beta_1,\cdots,\beta_n)=0.$$ It therefore suffices to show that $F_N(\beta_1,\cdots,\beta_n)$ is a polynomial in $\lambda$ that is not identically zero.  Observe that the expression on the right-hand side of (\ref{FN})
		is a symmetric polynomial in $\beta_1,\cdots,\beta_n$. Indeed, let $\tau$ be a permutation of $\{1,\cdots,n\}$. Then
		$$(\tau F_N)(\beta_1,\cdots,\beta_n)=F_N(\beta_{\tau(1)},\cdots,\beta_{\tau(n)}),$$
		that is,
		$$\tau\left(\prod_{\substack{A\subset\{1,\cdots,n\}\\ |A|=N}}
		\det\bigl[h(\beta_i,\beta_j)\bigr]_{i,j\in A}\right)=\prod_{\substack{A\subset\{1,\cdots,n\}\\ |A|=N}}
		\det\bigl[h(\beta_{\tau(i)},\beta_{\tau(j)})\bigr]_{i,j\in A}.$$
		Let $S_A$ denote the set of all permutations of $A$. Setting $j = \tau(i)$ and $\eta = \tau\sigma\tau^{-1}$, we have $\mathrm{sgn}(\eta) =\mathrm{sgn}(\sigma)$, and hence the definition of the determinant gives
		\begin{align*}
			\det\bigl[h(\beta_{\tau(i)},\beta_{\tau(j)})\bigr]_{i,j\in A}&=\sum\limits_{\sigma\in S_A}\mathrm{sgn}(\sigma)\prod\limits_{i\in A}h(\beta_{\tau(i)},\beta_{\tau\sigma(i)})\\
			&=\sum\limits_{\eta\in S_{\tau(A)}}\mathrm{sgn}(\eta)\prod\limits_{j\in \tau(A)}h(\beta_{j},\beta_{\eta(j)})\\
			&=\det\bigl[h(\beta_i,\beta_j)\bigr]_{i,j\in \tau(A)}.
		\end{align*}
		
		Therefore, we get that
		\begin{align*}
			(\tau F_N)(\beta_1,\cdots,\beta_n)&=\tau\left(\prod_{\substack{A\subset\{1,\cdots,n\}\\ |A|=N}}\det\bigl[h(\beta_i,\beta_j)\bigr]_{i,j\in A}\right)\\
			&=\prod_{\substack{A\subset\{1,\cdots,n\}\\ |A|=N}}\det\bigl[h(\beta_i,\beta_j)\bigr]_{i,j\in \tau(A)}\\
			&=\prod_{\substack{\tau(A)\subset\{1,\cdots,n\}\\ |\tau(A)|=N}}\det\bigl[h(\beta_i,\beta_j)\bigr]_{i,j\in \tau(A)}\\
			&=F_N(\beta_1,\cdots,\beta_n).
		\end{align*}
		Using the fundamental theorem of symmetric polynomials, we have that there exists a polynomial $Q_N$ such that
		\[
		F_N(\beta_1,\cdots,\beta_n)
		=
		Q_N\bigl(e_1(\beta_1,\cdots,\beta_n),\cdots,
		e_n(\beta_1,\cdots,\beta_n)\bigr),
		\]
		where $e_1,\cdots,e_n$ are the elementary symmetric polynomials given by
		\[
		e_1(\beta_1,\cdots,\beta_n)
		=
		\sum_{j=1}^n\beta_j,
		\qquad
		e_2(\beta_1,\cdots,\beta_n)
		=
		\sum_{1\leqslant i<j\leqslant n}\beta_i\beta_j,
		\]
		and, more generally,
		\[
		e_l(\beta_1,\cdots,\beta_n)
		=
		\sum_{1\leqslant i_1<\cdots<i_l\leqslant n}
		\beta_{i_1}\cdots\beta_{i_l},
		\qquad 1\leqslant l\leqslant n.
		\]
		In particular,
		\[
		e_n(\beta_1,\cdots,\beta_n)
		=
		\prod_{j=1}^n\beta_j.
		\]
		
		Since
		\[
		p(z)
		=
		z^n+a_{n-1}z^{n-1}+\cdots+a_1z+a_0
		\]
		and  $\beta_1,\cdots,\beta_n$ are the zeros of $p-\lambda$, we have
		\[
		p(z)-\lambda
		=
		z^n+a_{n-1}z^{n-1}+\cdots+a_1z+(a_0-\lambda)
		=
		\prod_{j=1}^n(z-\beta_j).
		\]
		By Vieta's theorem,
		\[
		e_l(\beta_1,\cdots,\beta_n)
		=
		(-1)^l a_{n-l},
		\qquad 1\leqslant l<n,
		\]
		and
		\[
		e_n(\beta_1,\cdots,\beta_n)
		=
		(-1)^n(a_0-\lambda).
		\]
		Thus, the first $(n-1)$ elementary symmetric polynomials are independent
		of $\lambda$, while the last one is a polynomial of degree $1$ in
		$\lambda$. Consequently,
		\[
		F_N(\beta_1,\cdots,\beta_n)
		=
		Q_N\bigl(
		-a_{n-1},
		a_{n-2},
		\cdots,
		(-1)^{n-1}a_1,
		(-1)^n(a_0-\lambda)
		\bigr):=G_N(\lambda).
		\]
		It follows that
		\begin{equation}\label{eq:Fm-polynomial}
			F_N(\beta_1,\cdots,\beta_n)=G_N(\lambda)\in\mathbb{C}[\lambda],
		\end{equation}
		i.e., $F_N(\beta_1,\cdots,\beta_n)$ is a polynomial in $\lambda$.
		
		Next, we will prove that $G_N$ is not identically zero. To do this, we are going to show that $$G_N(t^n)=F_N(\beta_{1,t},\cdots,\beta_{n,t})\neq 0$$ for any large enough $t$, where \(\beta_{1,t},\cdots,\beta_{n,t}\) denote the zeros of \(p(z)-t^n\).  Noting that $$\frac{p(t\zeta)-t^n}{t^n}
		=\frac{1}{t^n}\prod_{j=1}^n(t\zeta-\beta_{j,t})
		=\prod_{j=1}^n\left(\zeta-\frac{\beta_{j,t}}t\right),$$
		we obtain  that the zeros of \(\frac{p(t\zeta)-t^n}{t^n}\) are precisely given by $$\frac{\beta_{1,t}}t,\cdots,\frac{\beta_{n,t}}{t}.$$
		On the other hand,
		$$\frac{p(t\zeta)-t^n}{t^n}=\zeta^n-1+\sum_{k=0}^{n-1}a_kt^{k-n}\zeta^k.$$
		Clearly, $\zeta^n-1$ has $n$ simple zeros denoted by $\omega_1,\cdots,\omega_n$. Choose  an $\epsilon\in (0, 1)$ sufficiently small such that the disks
		$$ B(\omega_1,\epsilon),\cdots,B(\omega_n,\epsilon) $$
		are pairwise disjoint. For all sufficiently large $t$, we have that $$\Big|\frac{p(t\zeta)-t^n}{t^n}-(\zeta^n-1)\Big|\leqslant\sum\limits_{k=0}^{n-1}|a_k|t^{k-n}(1+\epsilon)^k<\sum\limits_{k=0}^{n-1}|a_k|t^{k-n}2^k<|\zeta^n-1|$$
		holds on each circle $\partial B(\omega_j,\epsilon)=\{\zeta\in\C: |\zeta-\omega_j|=\epsilon\}$, $j=1,\cdots,n$.
		According to Rouch\'e's theorem, the polynomial
		$\frac{p(t\zeta)-t^n}{t^n}$
		has exactly one zero in each disk $B(\omega_j,\epsilon)$ for all sufficiently large \(t\). It follows that
		\begin{equation}\label{eq:root-asymptotics}
			\left|\frac{\beta_{j,t}}{t}-\omega_j\right|<\epsilon\ \ \ \ (j=1,\cdots, n)
		\end{equation}
		for all sufficiently large $t$.
		
		From  (\ref{eq:h}), for $1\leqslant i,j\leqslant n$ we have
		$$h(\beta_{i,t},\beta_{j,t})=\sum_{k=1}^{n}\frac{a_k}{k+1}\sum_{\ell=0}^{k-1}(\ell+1)(k-\ell+1)\beta_{i,t}^{k-\ell-1}\beta_{j,t}^\ell.$$
		It follows that
		\begin{align}\label{eq:t^{1-n}h}
			\begin{split}
				t^{1-n}h(\beta_{i,t},\beta_{j,t}\bigr)&=
\sum\limits_{\ell=0}^{n-1}\frac{(\ell+1)(n-\ell+1)}{n+1}
\left(\frac{\beta_{i,t}}{t}\right)^{n-\ell-1}\left(\frac{\beta_{j,t}}{t}\right)^\ell\\
				&\qquad + \sum\limits_{k=1}^{n-1}\sum\limits_{\ell=0}^{k-1}\frac{a_k(\ell+1)(k-\ell+1)}{k+1}\left(\frac{\beta_{i,t}}{t}\right)^{k-\ell-1}\left(\frac{\beta_{j,t}}{t}\right)^\ell\frac{1}{t^{n-k}}.
			\end{split}		
		\end{align}
		Then we have by (\ref{eq:root-asymptotics}) that the second term in \eqref{eq:t^{1-n}h} tends to zero as $t$ tends to infinity. Set $$b_\ell:=\frac{(\ell+1)(n-\ell+1)}{n+1}, \ \ \ \ \ell=0, 1, \cdots, n-1.$$ Then
		\begin{equation}\label{eq:Mlimit}
			t^{1-n}h\bigl(\beta_{i,t},\beta_{j,t}\bigr)
			\rightarrow
			M_{ij}:=
			\sum_{\ell=0}^{n-1}
			b_\ell\omega_i^{n-\ell-1}\omega_j^\ell\ \ \ \ (t\to \infty).
		\end{equation}
		
		Fix $A=\{n_1,\cdots,n_N\}\subset\{1,\cdots,n\}$ with $|A|=N$.  Since
		$\omega_i^n=1$ and $|\omega_i|=1$, for every nonzero vector
		$x=(x_{n_1}, x_{n_2}, \cdots, x_{n_N})\in\C^N$ we have
		\begin{align}\label{positive}
\begin{split}
			\sum_{i,j\in A}\overline{x}_i\,\omega_iM_{ij}x_j
			&=\sum_{\ell=0}^{n-1}b_\ell\sum_{i,j\in A}\overline{x}_i\,\omega_i^{\,n-\ell}\omega_j^\ell x_j\\
			& =\sum_{\ell=0}^{n-1}b_\ell\sum_{i,j\in A}\overline{x}_i\,\overline{\omega}_i^{\,\ell}\omega_j^\ell x_j\\
			&=\sum_{\ell=0}^{n-1}b_\ell
			\bigg|\sum_{j\in A}x_j\omega_j^\ell\bigg|^2.
\end{split}
		\end{align}
		We claim that the last term  in (\ref{positive}) is  strictly positive. In fact, if all the sums vanished, then in
		particular
		$$\sum_{j\in A}x_j\omega_j^\ell=0 \ \ \ \   (\ell=0,\cdots, n-1),$$
since $b_{\ell}>0$ for all $\ell$.
Taking $\ell=0,\cdots,N-1$, the above equality can be regarded as a system of linear equations:
		$$
		\begin{bmatrix}
			1 & 1 & \cdots & 1\\
			\omega_{n_1} & \omega_{n_2} & \cdots & \omega_{n_N}\\
			\vdots & \vdots & & \vdots\\
			\omega_{n_1}^{N-1} & \omega_{n_2}^{N-1}
			& \cdots & \omega_{n_N}^{N-1}
		\end{bmatrix}
		\begin{bmatrix}
			x_{n_1}\\
			x_{n_2}\\
			\vdots\\
			x_{n_N}
		\end{bmatrix}
		=0.
		$$
		The coefficient matrix is a Vandermonde matrix, and its determinant satisfies that
		$$\prod_{1\leqslant r<s\leqslant N}(\omega_{n_s}-\omega_{n_r})\neq0,$$
		since $\omega_{n_1},\cdots,\omega_{n_N}$ are pairwise distinct. This yields that $x=0$, which is  a contradiction. Consequently,
		$$\sum_{i,j\in A}\overline{x}_i\,\omega_iM_{ij}x_j=\sum_{\ell=0}^{n-1}b_\ell\bigg|\sum_{j\in A}x_j\omega_j^\ell\bigg|^2>0$$
for every nonzero vector
		$x=(x_{n_1}, x_{n_2}, \cdots, x_{n_N})\in\C^N$.
		Thus, $[\omega_iM_{ij}]_{i,j\in A}$ is a positive definite matrix, and its determinant is strictly positive. Since this matrix is obtained from \([M_{ij}]_{i,j\in A}\) by multiplying its \(i\)-th row by \(\omega_i\) for each \(i\in A\), we have
		$$0<\det[\omega_iM_{ij}]_{i,j\in A}=\left(\prod_{i\in A}\omega_i\right)\det[M_{ij}]_{i,j\in A}.$$
		It follows that
		\begin{equation*}\label{eq:M-nonsingular}
			\det[M_{ij}]_{i,j\in A}\neq0.
		\end{equation*}
Combining this and \eqref{eq:Mlimit}, we conclude that $$\mathrm{det}\big[t^{1-n}h(\beta_{i,t},\beta_{j,t})\big]_{i, j\in A}\neq 0$$
for all sufficiently large $t$.
It follows from (\ref{FN}) and (\ref{eq:Fm-polynomial}) that $$G_N(t^n)=F_N(\beta_{1,t},\cdots,\beta_{n,t})\neq 0$$
when $t$ is large enough. This implies that $G_N$  is not identically zero, as desired.

In summary, we have proved that \(\Lambda_p\) is finite. Since \(p(\partial\mathbb D)\) is closed and disjoint from \(\Lambda_p\), each \(\lambda\in\Lambda_p\) is isolated in \(\sigma(\mathbb T_p)\) and is an eigenvalue of finite geometric multiplicity. Hence \(\Lambda_p\subset\pi_{00}(\mathbb T_p)\), and the reverse inclusion established above yields equality. Therefore, we arrive at
\[
\omega(\mathbb T_p)=\sigma(\mathbb T_p)\backslash\pi_{00}(\mathbb T_p),
\]
which implies that Weyl's theorem holds for $\T_p$.

To finish the proof, it remains to show that $$|\Lambda_p|\leqslant (n-1)2^{n-1}.$$
For this purpose, we  split the points of \(\Lambda_p\) into two classes.

First, consider those \(\lambda\in\Lambda_p\) for which \(p-\lambda\) has a multiple zero. If \(\gamma\) is such a multiple zero, then
\[
p(\gamma)=\lambda,\qquad p'(\gamma)=0.
\]
Since \(p'\) is a polynomial of degree $(n-1)$, it has at most $(n-1)$ distinct zeros. Hence, all such \(\lambda\) belong to the finite set
\[
\{p(\gamma):p'(\gamma)=0\},
\]
whose cardinality is at most $(n-1)$. Therefore this class contributes at most $(n-1)$ points to \(\Lambda_p\).

Now consider those \(\lambda\in\Lambda_p\) for which all zeros of \(p-\lambda\) are simple. As before, we use  \(N\) to denote the number of zeros of \(p-\lambda\) in \(\mathbb D\). From the preceding discussion, the cases \(N=0\) and \(N=n\) cannot occur. Thus, $
1\leqslant N\leqslant n-1$.
Fix such an \(N\). Let \(\lambda\) be a point of this type, and let
$\beta_1,\cdots,\beta_n$
be all the zeros of \(p-\lambda\). Since all these zeros are simple, they are pairwise distinct. Denote the \(N\) zeros lying in \(\mathbb D\)  by
$\alpha_1,\cdots,\alpha_N$.


We now estimate the degree of \(G_N\) defined in (\ref{eq:Fm-polynomial}). Since \(h(z,w)\) has total degree at most $(n-1)$, each determinant
$\det\bigl[h(\beta_i,\beta_j)\bigr]_{i,j\in A}$
has total degree at most \(N(n-1)\) in the variables \(\beta_1,\cdots,\beta_n\). Since \(F_N\) is the product of \(\binom nN\) such determinants, we get that
\[
\deg (F_N)
\leqslant
N(n-1)\binom nN
=
n(n-1)\binom{n-1}{N-1},
\]
where \(\binom kj\) denotes the binomial coefficient $\frac{k!}{j!(k-j)!}$.
Setting
\[
M_N:=(n-1)\binom{n-1}{N-1},
\]
we have
\[
\deg (F_N) \leqslant nM_N.
\]
If \(Q_N\) has degree \(d_N\) in the variable \(e_n\), then the term \(e_n^{d_N}\) contributes total degree at least \(nd_N\) in the variables \(\beta_1,\cdots,\beta_n\). Hence
\[
nd_N\leqslant \deg (F_N) \leqslant nM_N,
\]
so \(d_N\leqslant M_N\). Since only \(e_n=(-1)^n(a_0-\lambda)\) depends on \(\lambda\), it follows that
\[
\deg (G_N) \leqslant d_N \leqslant M_N.
\]
Therefore, \(G_N\) has at most \(M_N\) distinct zeros.

Thus, for fixed \(N\in\{1,\cdots,n-1\}\), the number of points \(\lambda\in\Lambda_p\) for which \(p-\lambda\) has only simple zeros and exactly \(N\) zeros in \(\mathbb D\) is at most
\[
M_N=(n-1)\binom{n-1}{N-1}.
\]
Combining this with the contribution from the multiple-zero case, we arrive at
\[
\begin{aligned}
|\Lambda_p|
&\leqslant
(n-1)+\sum_{N=1}^{n-1}(n-1)\binom{n-1}{N-1}\\
&=
(n-1)\left[1+\sum_{N=1}^{n-1}\binom{n-1}{N-1}\right]\\
&=
(n-1)\left[1+\sum_{k=0}^{n-2}\binom{n-1}{k}\right]\\
&=
(n-1)\bigl[1+(2^{n-1}-1)\bigr]\\
&=
(n-1)2^{n-1},
\end{aligned}
\]
which  completes the proof of Theorem \ref{thm:main}.
%
\end{proof}
\vspace{3mm}

\subsection*{Acknowledgment}
This work was supported by the National Natural Science Foundation of China
(12371125)  and  Chongqing Natural Science Foundation (CSTB2024NSCQ-MSX0177).

	\end{document}